\documentclass{elsarticle}

\usepackage[utf8]{inputenc}
\usepackage[english]{babel}
\usepackage{hyperref}
\usepackage{amsmath,amsthm}
\usepackage{enumerate}
\usepackage{enumitem}
\usepackage{amssymb}
\usepackage{graphicx, color}
\usepackage{epsfig}
\usepackage{tikz}
\usetikzlibrary{calc,decorations.pathmorphing,decorations.text}
\usepackage{subcaption}
\usepackage{refcount}
\usepackage[hmargin=2.5cm,vmargin=3cm]{geometry}
\usepackage{mathtools, braket}
\usepackage{centernot}
\usepackage{cases}
\usepackage[nameinlink]{cleveref}
\usepackage{comment}

\theoremstyle{plain}
\newtheorem{thm}{Thm}

\newtheorem{claim}{Claim}
\newtheorem{theorem}[thm]{Theorem}
\newtheorem{lemma}[thm]{Lemma}
\newtheorem{corollary}[thm]{Corollary}
\newtheorem{proposition}[thm]{Proposition}
\newtheorem{conjecture}[thm]{Conjecture}

\newtheorem{observation}[thm]{Observation}
\newtheorem{definition}[thm]{Definition}

\newtheorem{innercustomgeneric}{\customgenericname}
\providecommand{\customgenericname}{}
\newcommand{\newcustomtheorem}[2]{%
	\newenvironment{#1}[1]
	{%
		\renewcommand\customgenericname{#2}%
		\renewcommand\theinnercustomgeneric{##1}%
		\innercustomgeneric
	}
	{\endinnercustomgeneric}
}
\newcustomtheorem{customtheorem}{Theorem}
\newcustomtheorem{customlemma}{Lemma}

\newenvironment{proof*}{\noindent\emph{Proof of the claim:}}{\hfill$\Diamond$}

\renewcommand{\pod}[1]{\allowbreak\mathchoice
	{\if@display \mkern 0mu\else \mkern 0mu\fi (#1)}
	{\if@display \mkern 0mu\else \mkern 0mu\fi (#1)}
	{\mkern 1mu(\mathrm{mod}\mkern 4mu #1)}
	{\mkern 0mu(#1)}
}

\usepackage[color=green!30]{todonotes}

\usepackage{caption}
\usetikzlibrary{matrix}
\usetikzlibrary{decorations.markings}
\tikzstyle{vertex}=[circle, draw, fill=black!50,
inner sep=0pt, minimum width=4pt]
\tikzset{->-/.style={decoration={
			markings,
			mark=at position .5 with {\arrow{>}}},postaction={decorate}}}

\tikzstyle{bigblue}=[color=blue, very thick, >=stealth]
\tikzstyle{lightblue}=[color=blue, thin, >=stealth]

\tikzstyle{bigred}=[color=red, very thick, >=stealth]
\tikzstyle{lightred}=[color=red, thin, >=stealth]

\tikzstyle{biggreen}=[color=black!30!green, very thick, >=stealth]
\tikzstyle{lightgreen}=[color=black!30!green,  thin, >=stealth]

\begin{document}
	
		
	\begin{frontmatter}

	\title{Fractional balanced chromatic number of signed subcubic graphs}
		
\author[1]{Xiaolan Hu}
\author[2]{Luis Kuﬀner}
\author[3]{Jiaao Li}
\author[4]{Reza Naserasr}
\author[5]{Lujia Wang}
\author[3]{Zhouningxin Wang}
\author[6]{Xiaowei Yu}

\address[1]{School of Mathematics and Statistics \& Hubei Key Laboratory of Mathematical Sciences, Central China Normal University, Wuhan 430079, China.}
\address[2]{ \'{E}cole normale sup\'{e}rieure, PSL University, Paris, France}
\address[3]{ School of Mathematical Sciences and LPMC, Nankai University, Tianjin 300071, China}
\address[4]{ Universit\'{e} Paris Cit\'{e}, CNRS, IRIF, F-75013, Paris, France}
\address[5]{ Zhejiang Normal University, Jinhua, China}
\address[6]{ Jiangsu Normal University, Xuzhou, China}
\address[7]{ Emails: xlhu@mail.ccnu.edu.cn; \{lijiaao,  wangzhou\}@nankai.edu.cn; luis.kuffner.pineiro@ens.psl.eu;
	reza@irif.fr; ljwang@zjnu.edu.cn; xwyu@jsnu.edu.cn}

		\begin{abstract}
				A signed graph is a pair $(G,\sigma)$, where $G$ is a graph and $\sigma: E(G)\rightarrow \{-, +\}$, called signature, is an assignment of signs to the edges. Given a signed graph $(G,\sigma)$ with no negative loops, a balanced $(p,q)$-coloring of $(G,\sigma)$ is an assignment $f$ of $q$ colors to each vertex from a pool of $p$ colors such that each color class induces a balanced subgraph, i.e., no negative cycles.
			Let $(K_4,-)$ be the signed graph on $K_4$ with all edges being negative. In this work, we show that every signed (simple) subcubic graph admits a balanced $(5,3)$-coloring except for $(K_4,-)$ and signed graphs switching equivalent to it. For this particular signed graph the best balanced colorings are $(2p,p)$-colorings.			 
		\end{abstract}
		
		
	\end{frontmatter}

	\section{Introduction}
	Let $a,b$ be positive integers with $a\geq b$. First introduced in \cite{HRS73}, a \emph{fractional $\frac{a}{b}$-coloring} of a graph $G$ is an assignment $f: V(G) \to \binom{[a]}{b}$ where $[a]:=\{1,2,\ldots,a\}$ is a set of colors, such that $f(u)\cap f(v)=\emptyset$ for every edge $uv$ of $G$. The \emph{fractional chormatic number} of $G$, denoted $\chi_f(G)$, is defined to be $\chi_f(G)=\min \{\frac{a}{b}\mid \text{$G$ admits a fractional $\frac{a}{b}$-coloring}\}$. It is easy to see that $\chi_f(G)\leq \chi(G)$ as any proper $k$-coloring can be viewed as a fractional $\frac{k}{1}$-coloring. Let $\Delta(G)$ denote the maximum degree of $G$. By Brook's theorem, for $\Delta(G)\geq 3$, if $G$ contains no complete graph $K_{\Delta}$, then $\chi_f(G)\leq \Delta$.

	The fractional chromatic number of subcubic graphs (i.e., $\Delta\leq 3$) receives a great deal of attention. In particular, based on the study of independent sets in triangle-free subcubic graphs \cite{F78,GM96,HT01,S79}, Heckman and Thomas~\cite{HT01} conjectured that subcubic triangle-free graphs have fractional chromatic number at most $\frac{14}{5}$ and this bound is tight. Progress towards this conjecture can be found in~\cite{FKK14,HZ10,KKV11,L14,LP12}. It is resolved by Dvo\v r\'ak, Sereni, and Volec in \cite{DSV14}.  
	
	\begin{theorem}{\em \cite{DSV14}}
		Every subcubic triangle-free graph $G$ satisfies $\chi_f(G)\leq \frac{14}{5}$.
	\end{theorem}
	
	Recently, Dvo\v r\'ak, Lidick\'y and Postle \cite{DLP22} showed that every subcubic triangle-free graph avoiding two exceptional graphs as subgraphs admits a fractional $\frac{11}{4}$-coloring. This implies that every subcubic triangle-free planar graph has fractional chromatic number at most $\frac{11}{4}$. However, another conjecture of Heckman and Thomas~\cite{HT06} asserts that every subcubic triangle-free planar graph admits a fractional $\frac{8}{3}$-coloring which remains open. 
	
	\medskip
	Following this line of study, we explore the fractional balanced coloring of signed subcubic graphs in this paper. 
	
	A \emph{signed graph} $(G,\sigma)$ is a graph $G=(V,E)$ endowed with a \emph{signature function} $\sigma: E(G)\rightarrow \{-, +\}$ which assigns to each edge $e$ a sign $\sigma(e)$. An edge $e$ is called a \emph {positive edge} (or \emph{negative edge}, respectively) if $\sigma(e) = +$ (or $\sigma(e) = -$, respectively). The graph $G$ is called the \emph{underlying graph} of $(G, \sigma)$. When the signature is clear from context, $\widehat{G}$ is also used to denote a signed graph with the underlying graph $G$.
	
	Assume $(G, \sigma)$ is a signed graph and $v$ is a vertex of $G$. The \emph{vertex switching} at $v$ results in a signature $\sigma'$ defined as 
	
	$$\sigma'(e) = \begin{cases}
		- \sigma(e), &\text{if $v$ is a vertex of $e$ and $e$ is not a loop}; \cr 
		\sigma(e), &\text{ otherwise}.
	\end{cases}$$
	
	Two signatures $\sigma_1$ and $\sigma_2$ on the same underlying graph $G$ are said to be \emph{switching equivalent}, denoted by $\sigma_1\equiv\sigma_2$, if one is obtained from the other by a sequence of vertex switchings.

	Given a graph $G$, we denote by $(G,+)$ ($(G,-)$, respectively) the signed graph whose signature function is constantly positive (negative, respectively) on $G$. A signed graph $(G, \sigma)$ is {\em balanced} if $(G, \sigma) \equiv (G,+)$. A subset $X$ of vertices of a signed graph $(G, \sigma)$ is called {\em balanced} if $(G[X], \sigma|_{G[X]})$ is balanced. The size of a largest balanced set in $(G,\sigma)$ is denoted by $\beta(G,\sigma)$.

	\begin{definition}\label{def:fractionalBalancedColoring}
		Let $\widehat{G}$ be a signed graph. 
		Given a positive integer $p$ and a mapping $\phi:V(G)\to [p]$, a \emph{balanced $(p, \phi)$-coloring} of $\widehat{G}$, or simply a \emph{$(p, \phi)$-coloring} of $\widehat{G}$, is an assignment $f$, which assigns to each vertex $v$ a set of $\phi(v)$ colors from the set $[p]$ of $p$ colors in such a way that for each color $i$ the set of vertices assigned color $i$ is balanced. 
		
		If $\phi$ is the constant mapping $\phi(v)=q$ for every $v\in V(\widehat{G})$, then we write \emph{$(p,q)$-coloring} in place of $(p, \phi)$-coloring.
	\end{definition}
	
	It is observed that a signed graph $\widehat{G}$ admits a $(p,q)$-coloring for some $p\geq q$ if and only if it contains no negative loops. Hence, we assume that all signed graphs mentioned in this paper satisfy this property. On the other hand, the presence of a positive loop at a vertex does not affect the $(p,q)$-colorability of a signed graph. So we will always assume the signed graphs considered here have a positive loop attached to each of its vertices. On the other hand, in this work we do not allow parallel edges
	
	\begin{definition}
		Given a signed graph $\widehat{G}$, the \emph{fractional balanced chromatic number}, denoted $\chi_{\text{\it fb}}(\widehat{G})$, is defined as $$\chi_{\text{\it fb}}(\widehat{G})=\inf \left\{\frac{p}{q} \mid  \widehat{G} \text{ admits a } (p,q)\text{-coloring}\right\}.$$
	\end{definition}
	
	It is easily observed that the fractional balanced chromatic number is invariant under vertex switching.
	
	For the case $q=1$ in \Cref{def:fractionalBalancedColoring}, the fractional balanced coloring is reduced to a \emph{balanced $p$-coloring} of $\widehat{G}$, a notion first studied by Zaslavsky \cite{Z87} under the terminology ``balanced partition". The balanced coloring has drawn more attention recently when Jimenez, McDonald, Naserasr, Nurse, and Quiroz \cite{JMNNQ24+} showed an equivalent formulation of the famous Hadwiger conjecture with the setting of signed graphs and balanced chromatic number. For general $p$ and $q$ Kuffner, Naserasr, Wang, Yu, Zhou, and Zhu \cite{KNWYZZ24+} defined the signed analogy of the Kneser graphs, which serve as the homomorphism targets for fractional balanced coloring, and studied their balanced chromatic number. The same group showed that Hedetniemi's conjecture holds for the fractional balanced chromatic number and the categorical product of signed graphs \cite{KNWYZZ25+}. The problem studied in the current work also follows this direction of research.
	
	Observe that, being balanced, the color class $i$ can be switched to induce only positive edges. Hence, in practice, we use the following refined definition: A \emph{$(p, \phi)$-coloring} of $(G, \sigma)$ is a mapping $f$ of vertices where each vertex $v$ is assigned a set of $\phi(v)$ colors from the set $\pm[p]:=\{ \pm 1, \pm 2,\ldots \pm p\}$ such that first of all $-f(v)\cap f(v) =\emptyset$, and secondly if $\sigma(uv)=+1$, then $-f(u)\cap f(v) =\emptyset$ and if $\sigma(uv)=-1$, then $f(u)\cap f(v) =\emptyset$. For future reference, we write
	
	$$\binom{[p]}{\pm q}:= \{A \mid A \text{ is a } q \text{-subset of } \pm[p] \text{ such that } -A\cap A=\emptyset \}.$$ 
	
	The first observation for $(p, \phi)$-colorings is that one can permute colors and switch the role of $i$ with $-i$. Thus we have the following.
	
	\begin{observation}\label{obs:FreeVertex}
		If a signed graph $(G,\sigma)$ admits a $(p, \phi)$-coloring, then a for a fixed vertex $v$, any set $A$ of $\phi(v)$ colors satisfying $-A\cap A=\emptyset$ can be selected as the color set of $v$. 
	\end{observation}
	
	As a follow up to this this observation we have:
	
	\begin{lemma}
		If each $2$-connected block of a signed graph $\widehat{G}$ admits a $(p,\phi)$-coloring, then $\widehat{G}$ itself admits $(p, \phi)$-coloring,
	\end{lemma}
	
	\begin{proof}
		Observe that given two (signed) graphs $\widehat{G_1}$ and $\widehat{G_2}$  on distinct sets of vertices, if we identify one vertex from each, we create no new cycle, and, hence, the resulting signed graph is balanced if and only if both $\widehat{G_1}$ and $\widehat{G_2}$ are balanced. Thus merging the colorings of both $\widehat{G_1}$ and $\widehat{G_2}$ at the identified vertex, which can be done thanks to Observation~\ref{obs:FreeVertex}, we have a coloring for the merged signed graph.
	\end{proof}
	
	Another observation of a similar flavor is that a bridge  (i.e., a cut edge) do not affect the coloring at all.
	
	\begin{lemma}\label{lem:edge-cut}
		If a connected signed graph $\widehat{G}$ contains a bridge $uv$, then any $(p,\phi)$-colorings $\widehat{G}-uv$ is also a $(p, \phi)$-coloring of $\widehat{G}$.
	\end{lemma}
	
	\begin{proof}
		That is simply because $uv$ belong to no cycle, in particular to no negative cycle.
	\end{proof}
	
	We denote by $K_4^{\bullet}$ the graph obtained from $K_4$ by subdividing one edge exactly once.  Let $\widehat{K}_4^{\bullet}$ be the signed graph on $K_4^{\bullet}$ where all edges except one edge incident to the vertex of degree $2$ are negative, see \Cref{fig:K4Bullet}. In the figures, negative edges are in red and solid line, while positive edges are in blue and dashed.
	
	\begin{figure}[!htbp]
		\centering
		\begin{subfigure}[t]{.45\textwidth}
			\centering
			\begin{tikzpicture}[scale=.45]		
				\draw [line width=0.4mm, dotted, blue] (0,3) to (6,0);
				\draw [line width=0.4mm, red] (0,3) to (-3,0) to (3,0) to (0,3);
				\draw [line width=0.4mm, red] (-3,0) to (0,-3) to (6,0); 
				\draw [line width=0.4mm, red] (-3,0) to (0,-3) to (3,0); 
				
				\draw [fill=white,line width=0.5pt] (0,3) node[above] {$x$} circle (4pt);  
				\draw [fill=white,line width=0.5pt] (3,0) node[right=0.5mm] {$y$} circle (4pt); 
				\draw [fill=white,line width=0.5pt] (6,0) node[right=0.5mm] {$t$} circle (4pt); 
				\draw [fill=white,line width=0.5pt] (-3,0) node[left=0.5mm] {$z$} circle (4pt); 
				\draw [fill=white,line width=0.5pt] (0,-3) node[below] {$w$} circle (3.5pt); 
			\end{tikzpicture}
			\caption{$\widehat{K}_4^{\bullet}$}
			\label{fig:K4Bullet}	
		\end{subfigure}
		\begin{subfigure}[t]{.45\textwidth}
			\centering
			\begin{tikzpicture}[scale=.45]
				\draw [line width=0.4mm, red] (0,3) to (-3,0) to (3,0) to (0,3);
				\draw [line width=0.4mm, red] (-3,0) to (0,-3); 
				\draw [line width=0.4mm, red] (-3,0) to (0,-3) to (3,0); 
				\draw [line width=0.4mm, red] (0,3) to (0,-3); 

				\draw [fill=white,line width=0.5pt] (0,3) node[above] {$x$} circle (4pt);  
				\draw [fill=white,line width=0.5pt] (3,0) node[right=0.5mm] {$y$} circle (4pt); 
				\draw [fill=white,line width=0.5pt] (-3,0) node[left=0.5mm] {$z$} circle (4pt); 
				\draw [fill=white,line width=0.5pt] (0,-3) node[below] {$w$} circle (3.5pt); 
			\end{tikzpicture}
			\caption{$(K_4, -)$}
			\label{fig:K4}
		\end{subfigure}
		\caption{Subcubic graphs $\widehat{K}_4^{\bullet}$ and $(K_4, -)$}
		\label{fig:example-nail}
	\end{figure}
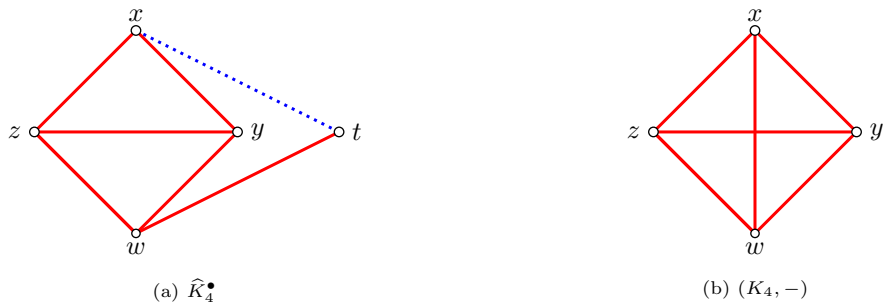

	In this work, we focus on the signed subcubic graphs and prove the following main result. 
	
	\begin{theorem}\label{thm:fractional-main}
		Every signed subcubic graph $\widehat{G}$ not switching equivalent to $(K_4,-)$ admits a $(5, 3)$-coloring. 
	\end{theorem}
	
	This implies that the fractional balanced chromatic number of any signed subcubic graph is at most $\frac{5}{3}$, except for $(K_4,-)$. The bound of $\frac{5}{3}$ is tight and is achieved by the signed graph $\widehat{K}_4^{\bullet}$. A proof of $\chi_{\text{\it fb}}(\widehat{K}_4^{\bullet})=5/3$ is given in~\Cref{lem:5/3}.

	To prove \Cref{thm:fractional-main}, we will prove a stronger statement (\Cref{thm:main} below) for which we need the following notions. A {\em block} in a graph $G$ is a maximal $2$-connected subgraph of $G$. Let $C^*_3$ denote a triangle $xyz$ of $G$ such that $d_G(x)=d_G(y)=2,$ and $d_G(z)\leq 3$. Similarly, let $C^*_4$ denote a $4$-cycle $xyzw$ of $G$ such that $d_G(x)=d_G(y)=d_G(z)=2,$ and $d_G(w)\leq 3$. 
	
	\begin{theorem}\label{thm:main}
		Let $\widehat{G}$ be a signed connected subcubic graph with no block of its underlying graph isomorphic to any graph in $\{C^*_3, C^*_4, K_4^{\bullet}, K_4\}$. Let $\phi(v)=6-d_G(v)$ for every $v\in V(G)$. Then $\widehat{G}$ admits a $(5, \phi)$-coloring. 
	\end{theorem}
	
	The rest of this paper is organized as follows. We give basic properties of fractional balanced colorings of signed subcubic graphs in \Cref{sec:Pre]}. In \Cref{sec:ProofFrac}, we prove \Cref{thm:fractional-main} using \Cref{thm:main}. \Cref{sec:ProofMain} is devoted to proving \Cref{thm:main}. Some remarks and further discussion are provided in \Cref{sec:Que}.

	\section{Preliminaries}\label{sec:Pre]}
	
	Some of the basic properties we will need are the following. Let $S=\{s_1,s_2,\ldots, s_t\}$ be a set of integers. Denote by $S^*$ the set of absolute values of the elements in $S$. For example, if $S=\{1, 2, 3, -3\}$, then $S^*=\{1, 2, 3\}$.  We denote by $|S|$ the cardinality of $S$. Clearly, $|S^*|\le |S|$.
	
	\begin{observation}\label{obs:proper_subset}
		Given a positive integer $p$, let $k_1$ and $k_2$ be two positive integers such that $2\leq k_i\leq p$ for $i\in [2]$. For any two sets $A_i\in {[p] \choose \pm k_i}$ for $i\in [2]$, there exists a proper subset $B_i\subsetneq A_i$ such that $B_i\in {[p] \choose \pm (k_i-1)}$ and $B_1^*\neq B_2^*$. 
	\end{observation}
	
	\begin{observation}\label{lem:InclusionColoring}
		Let $\widehat{G}$ be a signed graph and let $\phi_i: V(\widehat{G})\to [p]$ for $i\in [2]$ such that $\phi_2(v)\leq \phi_1(v)$ for every $v\in V(\widehat{G})$. If $\widehat{G}$ admits a $(p, \phi_1)$-coloring $f_1$,  then $\widehat{G}$ admits a $(p, \phi_2)$-coloring $f_2$ such that $f_2(v)\subseteq f_1(v)$ for each $v\in V(\widehat{G})$.
	\end{observation}

	\begin{observation}\label{obs:subgraph}
		Let $\widehat{G}$ and $\widehat{H}$ be signed graphs with a homomorphism $\psi$ from $\widehat{G}$ to $\widehat{H}$. Assume $\widehat{H}$ admits a $(p, \phi)$-coloring for a given integer $p$ and a mapping $\phi:V(G)\to [p]$. Then $\widehat{G}$ admits a $(p, \phi')$-coloring where $\phi'(u)=\phi(\psi(u))$.  
	\end{observation}

	\begin{proposition}{\em \cite{KNWYZZ25+}}\label{prop:C-k}
		For each positive integer $k$ with $k\geq 2$, the negative cycle $C_{-k}$ admits a $(k, k-1)$-coloring. 
	\end{proposition}
	
	In particular, by removing colors $\pm 6, \ldots, \pm k$ when $k\geq 5$, we have:
	
	\begin{corollary}\label{lem:5+-cycle}
		For each positive integer $k$ with $k\geq 5$, the negative cycle $C_{-k}$ admits a $(5, 4)$-coloring. 
	\end{corollary}
	
	Every balanced signed graph $(G,+)$ admits a $(k,k)$-coloring for any positive integer $k$, in particular, $\chi_{\text{\it fb}}(G,+)=1$.

	\begin{lemma}\label{lem:smalCase}
		Let $\phi_3$ be an assignment of integer $3$, $3$, and $4$ to the three vertices of $C_3$ and let $\phi_4$ be an assignment of $3$, $3$, $4$, and $4$ to the vertices of $C_4$. We then have the following claims.
		
		\begin{itemize}
			\item $(C_3, \sigma)$ admits a $(5,\phi_3)$-coloring for any $\sigma$.
			\item $(C_4, \sigma)$ admits a $(5,\phi_4)$-coloring for any $\sigma$.
			\item $(K_4^{\bullet}, \sigma)$ admits a $(5,3)$-coloring for any $\sigma$.
		\end{itemize} 
	\end{lemma}
	
	\begin{proof}
		First notice that every balanced signed graph admits a $(5,5)$-coloring, which is stronger than requested $(5,\phi)$-colorings in each case. So, it is enough to consider signatures that induce some negative cycle.
		
		For $C_{-3}$ ($C_{-4}$, respectively)  we assign the color set $\{1,2,3,4\}$ to the vertrex $v$ with $\phi_3(v)=4$ (the vertices with $\phi_4(x)=4$, respectively), and color sets $\{1,2,5\}$ and $\{3,4,5\}$ to the other two vertices.
		
		For $K_4^\bullet$ by symmetry, beside the balanced case, there are three switching equivalent classes signed graphs: the one in Figure \ref{fig:K4Bullet}, and the two in Figure \ref{fig:K4bullet23}.
		
		For $\widehat{K}_4^\bullet$, we define $f(t)=\{3,4,5\}$, $f(x)=\{1,2,3\}$, $f(y)=\{-2,-4,-5\}$, $f(z)=\{-1,-3,5\}$, and $f(w)=\{1,2,4\}$. It can be easily checked that $f$ is a $(5,3)$-coloring.
		
		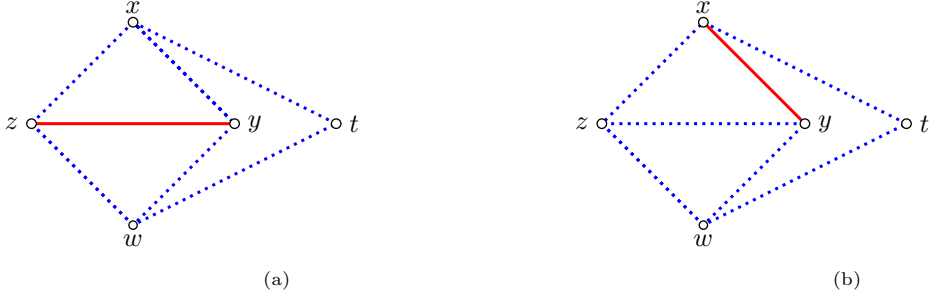
\begin{figure}[!htbp]
			\centering
			\begin{subfigure}[t]{.45\textwidth}
				\begin{tikzpicture}[scale=.45]		
					\draw [line width=0.4mm, dotted, blue] (0,3) to (6,0);
					\draw [line width=0.4mm, dotted, blue] (-3,0) to (0,3) to (3,0) to (0,3);
					\draw [line width=0.4mm, red] (-3,0) to (3,0);
					\draw [line width=0.4mm, dotted, blue] (-3,0) to (0,-3) to (6,0); 
					\draw [line width=0.4mm, dotted, blue] (-3,0) to (0,-3) to (3,0); 
					
					\draw [fill=white,line width=0.5pt] (0,3) node[above] {$x$} circle (4pt);  
					\draw [fill=white,line width=0.5pt] (3,0) node[right=0.5mm] {$y$} circle (4pt); 
					\draw [fill=white,line width=0.5pt] (6,0) node[right=0.5mm] {$t$} circle (4pt); 
					\draw [fill=white,line width=0.5pt] (-3,0) node[left=0.5mm] {$z$} circle (4pt); 
					\draw [fill=white,line width=0.5pt] (0,-3) node[below] {$w$} circle (3.5pt); 
				\end{tikzpicture}
				\caption{}
				\label{fig:K4bullet2}	
			\end{subfigure}
			\begin{subfigure}[t]{.45\textwidth}
				\begin{tikzpicture}[scale=.45]
					\draw [line width=0.4mm, dotted, blue] (0,3) to (6,0);
					\draw [line width=0.4mm, red](0,3) to (3,0);
					\draw [line width=0.4mm, dotted, blue] (3,0) to (-3,0) to (0,3);
					\draw [line width=0.4mm, dotted, blue] (-3,0) to (0,-3) to (6,0); 
					\draw [line width=0.4mm, dotted, blue] (-3,0) to (0,-3) to (3,0); 
					
					\draw [fill=white,line width=0.5pt] (0,3) node[above] {$x$} circle (4pt);  
					\draw [fill=white,line width=0.5pt] (3,0) node[right=0.5mm] {$y$} circle (4pt); 
					\draw [fill=white,line width=0.5pt] (6,0) node[right=0.5mm] {$t$} circle (4pt); 
					\draw [fill=white,line width=0.5pt] (-3,0) node[left=0.5mm] {$z$} circle (4pt); 
					\draw [fill=white,line width=0.5pt] (0,-3) node[below] {$w$} circle (3.5pt); 
				\end{tikzpicture}
				\caption{}
				\label{fig:K4bullet3}
			\end{subfigure}
			\caption{Two signed graphs on ${K}_4^{\bullet}$}
			\label{fig:K4bullet23}
		\end{figure}
		
		For the graph in Figure \ref{fig:K4bullet2}, contracting first $xt$ then $xw$ results in a $C_{-3}$ as a homomorphic image. Similarly, for the graph in Figure \ref{fig:K4bullet3}, contracting first $xt$ then $zw$ results in a $C_{-3}$ as a homomorphic image. Both are $(5,3)$-colorable and we are done.
	\end{proof}

	\begin{lemma}\label{lem:4vertices}
		Every signed graph $\widehat{G}$ on at most $4$ vertices which is not switching equivalent to $(K_4, -)$, admits a $(3,2)$-coloring.
	\end{lemma}
	
	\begin{proof}
		Any such signed graph is a subgraph of $(K_4,\sigma)$ that is not switching equivalent to $(K_4,-)$. Since each edge of $(K_4,\sigma)$ is in exactly two triangles, there exist two positive triangles sharing an edge. By possibly a switching, we may assume this edge is positive. Contracting this edge results in a homomorphic image which is either $C_{+3}$ or $C_{-3}$ both of which admit a $(3,2)$-coloring.
	\end{proof}
	
	\begin{lemma}\label{lem:5/3}
		$\chi_{\text{\it fb}}(\widehat{K}_4^{\bullet})=\frac{5}{3}$.
	\end{lemma}
	\begin{proof}
		The upper bound $\chi_{\text{\it fb}}(\widehat{K}_4^{\bullet})\leq \frac{5}{3}$ is already shown in \Cref{lem:smalCase}.
		
		Similar to graph cases, it was shown in \cite{KNWYZZ25+} that the fractional balanced chromatic number of a signed graph satisfies $$\chi_{\text{\it fb}}(G,\sigma)\geq \frac{|V(G)|}{\beta(G,\sigma)}$$
		Thus, noticing that the size of a maximum balanced set in $\widehat{K}_4^{\bullet}$ is 3, we obtain the desired lower bound.
	\end{proof}

	\subsection[ProofFrac]{The proof of \Cref{thm:fractional-main}}\label{sec:ProofFrac}

	We first derive \Cref{thm:fractional-main} from \Cref{thm:main}. 
	
	\medskip
	\noindent
	\emph{Proof of \Cref*{thm:fractional-main}.}
	Let $\widehat{G}$ be a subcubic signed graph on $n$ vertices which is a not $(K_4,-)$. We consider the following cases:
	
	\textbf{Case (1).} The underlying graph of $\widehat{G}$ is isomorphic to one of $C_3, C_4$ or $K_4^{\bullet}$. In this case, by \Cref{lem:smalCase}, we are done. 
	
	\textbf{Case (2).} ${G}$ contains no block isomorphic to any graph in $\{C^*_3, C^*_4, K_4^{\bullet}\}$. By~\Cref{thm:main}, $\widehat{G}$ admits a $(5, \phi)$-coloring $f$ where $\phi(v)=6-d_G(v)$. By \Cref{lem:InclusionColoring}, since $6-d_G(v)\geq 3$ for $G$ being subcubic, $\widehat{G}$ admits a $(5, 3)$-coloring.
	
	\textbf{Case (3).} ${G}$ contains a block $H$ whose underlying graph is in $\{C^*_3, C^*_4, K_4^{\bullet}\}$. We remove all vertices of $H$ except the one connecting it to the rest of $\widehat{G}$. We color the resulting graph by an induction hypothesis and then extend the coloring to the rest of $H$ by \Cref{lem:InclusionColoring}. \hfill \qed
	

	\subsection[extension]{Extension of partial $(5, \phi)$-colorings}
	
	We need a further preparation to complete the proof of \Cref{thm:main}. We first have the following observation.
	
	\begin{observation}\label{obs:1-vertex}
		Let $(G,\sigma)$ be a signed graph with a $1$-vertex $u$ whose neighbor is $v$ and $\sigma(uv)=-$. Let $\phi: V(G)\to[p]$ be a mapping and assume $f$ is a $(p,\phi)$-coloring of $(G,\sigma)-u$. Then for any $X\in \binom{[p]}{\pm\phi(u)}$ such that $X\cap f(v)=\emptyset$, the assignment $f(u)=X$ is an extension of $f$ to a $(p,\phi)$-coloring of $(G,\sigma)$.
		
	\end{observation}
	
	The following notion will be used. Let $f$ be a $(5, \phi)$-coloring of $\widehat{G}$. For any $1$-vertex $v$ and its neighbor $u$, we define the \emph{available color set} of $v$ with respect to $f$ by $A_f(v):=\pm [5] \setminus f(u)$. Here, $|A^*_f(v)|=5$ and $|A_f(v)|=10-\phi(u)$, in particular, if $\phi(u)=6-d_G(u)$, then $|A_f(v)|=4+d_G(u)$. 
	
	\section[MainProof]{Proof of~\Cref{thm:main}}\label{sec:ProofMain}
	In this section, we assume that $\widehat{G}$ is a counterexample to~\Cref{thm:main} with the number of vertices being minimized. That is to say, $\widehat{G}$ is a signed subcubic graph with no block isomorphic to any element of $\mathcal{B}_0:=\{C^*_3, C^*_4, K_4^{\bullet}, K_4\}$ that does not admit a $(5, \phi)$-coloring with $\phi(v)=6-d_G(v)$. By minimality of $\widehat{G}$, any signed subcubic graph $\widehat{H}$ with no block isomorphic to any member of $\mathcal{B}_0$ and has fewer vertices than $\widehat{G}$, admits a $(5, \phi')$-coloring with $\phi'(v)=6-d_H(v)$. We may assume that $\widehat{G}$ is connected. Furthermore, by~\Cref{lem:4vertices}, $\widehat{G}$ has at least $5$ vertices, and by~\Cref{lem:5+-cycle}, $\widehat{G}$ must contain at least one $3$-vertex.

	
	
	\begin{claim}\label{claim:2-connected}
		$\widehat{G}$ is $2$-connected. In particular, $\delta(\widehat{G})\ge 2$.
	\end{claim}
	
	\begin{proof}
		We first show that there is no vertex of degree 1. Suppose to the contrary and assume $v$ is a vertex of degree 1 with $u$ being its only neighbor. If $v$ is a vertex of degree 1, then it is in no cycle and it can be given all the $5$ colors without being involved in inducing a negative cycle. The key point is to show that $\widehat{G}-v$ admits the required coloring. If  $\widehat{G}-v$  satisfies the conditions of the theorem, i.e., if $\widehat{G}-v$ contains none of  $C^*_3, C^*_4, K_4^{\bullet}, K_4$ as a block, then we have a coloring by the minimality of $G$. Otherwise, either  $\widehat{G}-v$ is signed graph on $K_4^{\bullet}$ for which a $(5,3)$-coloring is required.  That is provided in \Cref{lem:smalCase}. Or, $\widehat{G}-v$ has one of $C^*_3, C^*_4$ as a block. In this case we note two facts: 1. there should be an a cut edge connecting this $C^*_3$, or $C^*_4$ to the rest of the graph, and 2. the neighbor $u$ of $v$ is a vertex of degree $3$ in $\widehat{G}$ and hence requires only $3$ colors. We can then apply \Cref{lem:smalCase} to complete the coloring on $v$ and the subgraph $C^*_3$ or  $C^*_4$ we are working with. Then we can use the cut edge and minimality $\widehat{G}$ to complete the coloring to the rest of the graphs. 
		
		If $v$ is not of degree 1, then, since $G$ is subcubic, having a vertex cut implies having an cut edge $uv$. Let $\widehat{G}_v$ and $\widehat{G}_u$ be the components of $\widehat{G}-uv$ containing $v$ and $u$ respectively. Adding the vertex $u$ to $\widehat{G}_v$ and $v$ to $\widehat{G}_u$ we have two proper subgraphs of $\widehat{G}$ each satisfying the conditions of the theorem, that is to say neither contains one of $C^*_3, C^*_4, K_4^{\bullet}, K_4$ as a block. They each then admit a $(5, \phi)$-coloring that can be put together thanks to \Cref{lem:edge-cut}.		
	\end{proof}
	
	We observe that the bad blocks that might be created after deleting or cutting are either $C^*_3$ or $C^*_4$. The next claim guarantees that after deletion, we shall not create a block isomorphic to $C^*_3$ or $C^*_4$. 
	
	A \emph{$t$-cycle} in a signed graph $G$ is referred to as a \emph{$(d_1,d_2,\ldots,d_t)$-cycle} if $d_G(v_i)=d_i$ for every $i$ with $i\in[t]$. 
	
	\begin{claim}\label{claim:3OR4cycles}
		$\widehat{G}$ contains none of the following: a $(2,3,3)$-triangle, a $(2,3,2,3)$-cycle, or a $(2,2,3,3)$-cycle.
	\end{claim}
	
	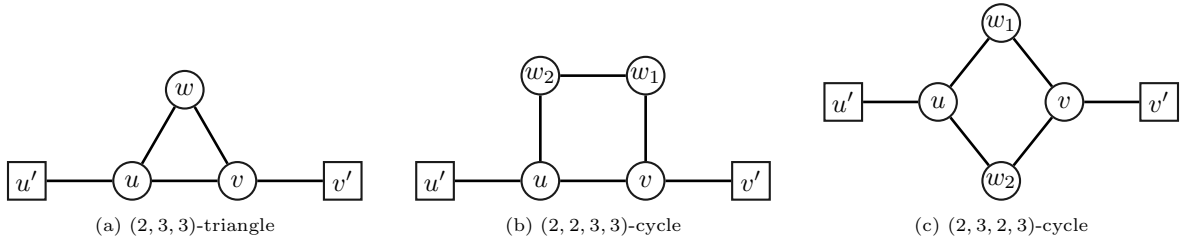
\begin{figure}[htbp]
		\centering  
		\begin{subfigure}[t]{.32\textwidth}
			\centering
			\begin{tikzpicture}[>=latex,
				roundnode/.style={circle, draw=black!90, thick, minimum size=5mm, inner sep=0pt},
				squarenode/.style={rectangle, draw=black!90, thick, minimum size=5mm, inner sep=0pt},
				scale=0.7
				]
				\node [roundnode] (v) at (1,0) {$v$};
				\node [roundnode] (u) at (-1,0) {$u$};
				\node [roundnode] (w) at (0,1.732) {$w$};
				\node [squarenode] (y) at (3, 0) {$v'$};
				\node [squarenode] (x) at (-3, 0) {$u'$};
				
				\draw [line width =1pt, black] (y)-- (v)--(u)--(x);
				\draw [line width =1pt, black] (u)--(w)--(v);
			\end{tikzpicture}
			\caption{$(2,3,3)$-triangle}  
			\label{fig:233-cycle}  
		\end{subfigure}
		\begin{subfigure}[t]{.32\textwidth}
			\centering
			\begin{tikzpicture}[>=latex,
				roundnode/.style={circle, draw=black!90, thick, minimum size=5mm, inner sep=0pt},
				squarenode/.style={rectangle, draw=black!90, thick, minimum size=5mm, inner sep=0pt},
				scale=0.7
				]
				\node [roundnode] (v) at (1,0) {$v$};
				\node [roundnode] (u) at (-1,0) {$u$};
				\node [roundnode] (w1) at (1,2) {$w_1$};
				\node [roundnode] (w2) at (-1,2) {$w_2$};
				\node [squarenode] (y) at (3, 0) {$v'$};
				\node [squarenode] (x) at (-3, 0) {$u'$};
				
				\draw [line width =1pt, black] (x)--(u)--(v)--(w1)--(w2)--(u);
				\draw [line width =1pt, black] (v)--(y);
			\end{tikzpicture}
			\caption{$(2,2,3,3)$-cycle}  
			\label{fig:2233-cycle} 
		\end{subfigure}
		\begin{subfigure}[t]{.32\textwidth}
			\centering
			\begin{tikzpicture}[>=latex,
				roundnode/.style={circle, draw=black!90, thick, minimum size=5mm, inner sep=0pt},
				squarenode/.style={rectangle, draw=black!90, thick, minimum size=5mm, inner sep=0pt},
				scale=0.7
				]
				\node [roundnode] (v) at (1.2,0) {$v$};
				\node [roundnode] (u) at (-1.2,0) {$u$};
				\node [roundnode] (w1) at (0,1.5) {$w_1$};
				\node [roundnode] (w2) at (0,-1.5) {$w_2$};
				\node [squarenode] (y) at (3, 0) {$v'$};
				\node [squarenode] (x) at (-3, 0) {$u'$};
				
				\draw [line width =1pt, black] (x)--(u)--(w1)--(v);
				\draw [line width =1pt, black] (u)--(w2)--(v)--(y);
			\end{tikzpicture}
			\caption{$(2,3,2,3)$-cycle}  
			\label{fig:2323-cycle} 
		\end{subfigure}
		\caption{Configurations in \Cref{claim:3OR4cycles}}
		\label{fig:cycles}
	\end{figure}  
	
	\begin{proof}
		Suppose, to the contrary, that $\widehat{G}$ contains one of the following: a $(2,3,3)$-triangle $wuv$, a $(2,2,3,3)$-cycle $w_1w_2uv$, or a $(2,3,2,3)$-cycle $w_1vw_2u$ where $d_G(u)=d_G(v)=3$ and $d_G(w)=d_G(w_i)=2$ for $i\in [2]$. See \Cref{fig:cycles}. Let $u'$ and $v'$ be the other neighbors of $u$ and $v$, respectively. Since $\widehat{G}$ is $2$-connected (\Cref{claim:2-connected}), $u'$ and $v'$ are distinct. By possibly a switching, we may assume that $uu',vv'$ are both negative and each edge of the cycle is negative except $uv$ for \Cref{fig:233-cycle} and \Cref{fig:2233-cycle} and $vw_2$ for \Cref{fig:2323-cycle}. Let $X=\{u,v\}$. 
		
		Let $\widehat{H}$ be the component of $\widehat{G}-X$ not containing $w$ or $w_i$ and let $\widehat{H}_X$ be the subgraph of $\widehat{G}$ by removing $w$ or both $w_i$'s and the edge $uv$. Note that $d_{H_X}(v)=d_{H_X}(u)=1$. Since $\widehat{H}_X$ is a proper subgraph of $\widehat{G}$ and, moreover, it contains no block isomorphic to any element of $\mathcal{B}_0$, hence admits a $(5, \phi')$-coloring $f$ where $\phi'(x)=6-d_{H_X}(x)$ for $x \in V(H_X)$. Furthermore, since $d_{H_X}(v)=d_{H_X}(u)=1$, $d_{H_X}(u')\geq 2$ and $d_{H_X}(v')\geq 2$ without loss of generality, assume that $\{c_1,c_2,c_3,c_4,\pm c_5\}\subseteq A_f(v)$ and $\{d_1,d_2,d_3,d_4, d_5,d_u\}\subseteq A_f(u)$ with $|c_k|=|d_k|=k$ for each $k\in[5]$ and $d_u\in \pm [5]\setminus \{d_1,d_2,d_3,d_4,d_5\}$. In each case, we define a new mapping $g$ as follows: 
		\begin{align*}
			g(v)&= 
			\begin{cases}
				\{c_1,c_2,-d_5\}, &\text{if $uv$ is negative};\\
				\{c_1,c_2,d_5\}, &\text{if $uv$ is positive or $u,v$ are not adjacent};
			\end{cases}\\
			g(u)&=\{d_3,d_4,d_5\}.
		\end{align*}
		Observe that $g(u)\subset f(u)$, $g(v)\subset f(v)$, and $g(u), g(v)\in {[5] \choose \pm 3}$. We shall extend such a coloring $g$ to be a $(5, \phi)$-coloring with $\phi(x)=6-d_G(x)$ for $x\in V(G)$.
		
		\begin{itemize}
			\item For the $(2,3,3)$-triangle $wuv$, recall that both of $wv$ and $wu$ are negative. Let $g(w)=\{-c_1,-c_2,-d_3,-d_4\}$.

			\item For the $(2,2,3,3)$-cycle $w_1w_2uv$, recall that $vw_1,w_1w_2,w_2u$ are all negative. Let $g(w_1)=\{-c_1,-c_2,d_3,d_4\}$, and $g(w_2)=\{c_1,c_2,-d_3,-d_4\}$.

			\item For the $(2,3,2,3)$-cycle $w_1vw_2u$, recall that edges $uw_1,vw_1,uw_2$ are negative. Since $u$ and $v$ are not adjacent in this case, $g(v)=\{c_1,c_2,d_5\}$ and $g(u)=\{d_3,d_4,d_5\}$. For vertices $w_1$ and $w_2$, we define that $g(w_1)=\{-c_1,-c_2,-d_3,-d_4\}$ and
			\begin{equation*}
				g(w_2)=
				\begin{cases}
					\{-c_1,-c_2,-d_3,-d_4\}, &\text{if $vw_2$ is negative};\\
					\{-c_1,-c_2,d_3,d_4\}, &\text{if $vw_2$ is positive}.
				\end{cases}
			\end{equation*}
		\end{itemize}
		In each case, it is easy to check that together with $g(x)=f(x)$ for $x\in V(H_X)\setminus \{u,v\}$, such $g$ is a $(5, \phi)$-coloring of $\widehat{G}$ with $\phi(x)=6-d_G(x)$, a contradiction. 
	\end{proof}
	
	In the following claims, we aim to show that there are no consecutive $2$-vertices in $\widehat{G}$.
	
	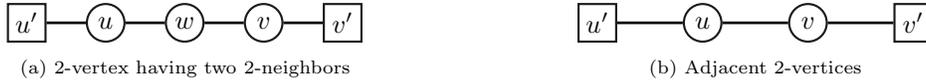
\begin{figure}[htbp]
		\centering  
		\begin{subfigure}[t]{.45\textwidth}
			\centering
			\begin{tikzpicture}[>=latex,
				roundnode/.style={circle, draw=black!90, thick, minimum size=5mm, inner sep=0pt},
				squarenode/.style={rectangle, draw=black!90, thick, minimum size=5mm, inner sep=0pt},
				scale=0.7
				]
				\node [roundnode] (v) at (1.5,0) {$v$};
				\node [roundnode] (u) at (-1.5,0) {$u$};
				\node [roundnode] (w) at (0,0) {$w$};
				\node [squarenode] (y) at (3, 0) {$v'$};
				\node [squarenode] (x) at (-3, 0) {$u'$};
				
				\draw [line width =1pt, black] (y)--(v)--(w)--(u)--(x);
			\end{tikzpicture}
			\caption{$2$-vertex having two $2$-neighbors}  
			\label{fig:2-2-2}  
		\end{subfigure}
		\begin{subfigure}[t]{.45\textwidth}
			\centering
			\begin{tikzpicture}[>=latex,
				roundnode/.style={circle, draw=black!90, thick, minimum size=5mm, inner sep=0pt},
				squarenode/.style={rectangle, draw=black!90, thick, minimum size=5mm, inner sep=0pt},
				scale=0.7
				]
				\node [roundnode] (v) at (1,0) {$v$};
				\node [roundnode] (u) at (-1,0) {$u$};
				\node [squarenode] (y) at (3, 0) {$v'$};
				\node [squarenode] (x) at (-3, 0) {$u'$};
				
				\draw [line width =1pt, black] (x)--(u)--(v)--(y);
			\end{tikzpicture}
			\caption{Adjacent $2$-vertices}  
			\label{fig:2-2} 
		\end{subfigure}
		\caption{Consecutive $2$-vertices in \Cref{claim:2-2-2} and \Cref{claim:2-2}}
		\label{fig:2-vertices}
	\end{figure}

	\begin{claim}\label{claim:2-2-2}
		$\widehat{G}$ contains no $2$-vertex whose two neighbors are both $2$-vertices.
	\end{claim}
	
	\begin{proof}
		Assume to the contrary that $\widehat{G}$ contains a $2$-vertex $w$ which has two $2$-neighbors $u$ and $v$, whose neighbors are $u'$ and $v'$, respectively, as depicted in \Cref{fig:2-2-2}. As $\widehat{G}$ contains no block isomorphic to $C^*_4$, we have that  $u'\neq v'$. By \Cref{claim:2-connected}, $u'$ and $v'$ cannot be in the same $(2,3,3,3)$-cycle or the same $(3,3,3)$-triangle in $\widehat{G}$. By possible switching, we may assume that $u'u, uw, wv, vv'$ are all negative.
		
		Let $\widehat{G'}:=\widehat{G}-\{u,w,v\}$. Note that no block isomorphic to $C^*_3$ or $C^*_4$ is created in $\widehat{G}'$, as otherwise either $u'$ or $v'$ would be in one of $(2,3,3)$-triangle, $(2,2,3,3)$-cycle or $(2,3,2,3)$-cycle in $\widehat{G}$, a contradiction to \Cref{claim:3OR4cycles}. By minimality, $\widehat{G}'$ admits a $(5, \phi')$-coloring $f$ where $\phi'(x)=6-d_{G'}(x)$ for $x \in V(G')$. Since $d_G(u')=d_{G'}(u')+1$ and $d_G(v')=d_{G'}(v')+1$, by definition $f(u')\in {[5] \choose \pm (7-d_G(u'))}$ and $f(v')\in {[5] \choose \pm (7-d_G(v'))}$. By \Cref{obs:proper_subset}, we can choose two subsets $f'(v')\subsetneq f(v')$ and $f'(u')\subsetneq f(u')$ satisfying that 
		$f'(v')\in {[5] \choose \pm (6-d_G(v'))}$, $f'(u')\in {[5] \choose \pm (6-d_G(u'))}$, and $f'(v')^*\neq f'(u')^*$. Let $A_{f'}(u)=\pm[5]\setminus f'(u')$ and $A_{f'}(v)=\pm[5]\setminus f'(v')$. We may assume, without loss of generality, that $\{c_1,c_2,c_3,c_4,\pm c_5\}\subseteq A_{f'}(u)$ and $\{d_1,d_2,d_3,\pm d_4,d_5\}\subseteq A_{f'}(v)$, where $|c_i|=|d_i|=i$ for every $i\in[5]$. 
		We define a new mapping $g$ as follows: $$g(u)=\{c_1,c_2,c_4,d_5\},~g(v)=\{d_1,d_3,c_4,d_5\},~g(w)=\{-c_2,-d_3,-c_4,-d_5\},$$ $g(u')=f'(u')$, $g(v')=f'(v')$, and $g(x)=f(x)$ for $x\in V(G)\setminus \{u', v', u,v,w\}$. Noting that $g(u)\subset A_{f'}(u)$ and $g(v)\subset A_{f'}(v)$, it is easy to see that $g$ is a $(5, \phi)$-coloring of $\widehat{G}$ with $\phi(x)=6-d_{G}(x)$, a contradiction. 
	\end{proof}

	\begin{claim}\label{claim:2-2}
		$\widehat{G}$ contains no adjacent $2$-vertices.
	\end{claim}
	
	\begin{proof}
		Assume to the contrary that $\widehat{G}$ contains two adjacent $2$-vertices $u$ and $v$, whose neighbors are $u'$ and $v'$, respectively, as depicted in \Cref{fig:2-2}. By \Cref{claim:2-2-2}, $d_G(u')=d_G(v')=3$. As $\widehat{G}$ contains no block isomorphic to $C^*_3$, we know that $u'\neq v'$. Similar as before, $u'$ and $v'$ cannot be in the same $(2,3,3,3)$-cycle or the same $(3,3,3)$-triangle in $\widehat{G}$. By possibly a switching, we may assume that $u'u, uv, vv'$ are all negative. 
		
		Let $\widehat{G'}:=\widehat{G}-\{u,v\}$. Using similar argument in the proof of ~\Cref{claim:2-2-2},    
		no block isomorphic to $C^*_3$ or $C^*_4$ is created in $\widehat{G'}$.
		By minimality, there exists a $(5, \phi')$-coloring $f$ of $\widehat{G'}$ where $\phi'(x)=6-d_{G'}(x)$ for $x \in V(G')$. By the definition $f(u')\in {[5] \choose \pm 4}$ and $f(v')\in {[5] \choose \pm 4}$. By \Cref{obs:proper_subset}, we can find $f'(u')\subsetneq f(u')$ and $f'(v')\subsetneq f(v')$ with $f'(u')\in {[5] \choose \pm 3}$ and $f'(v')\in {[5] \choose \pm 3}$ such that $f'(u')\neq f'(v')$. Without loss of generality, assume that $A_{f'}(u):=\pm [5]\setminus f'(u')=\{c_1,c_2,c_3,\pm c_4,\pm c_5\}$ and $A_{f'}(v):=\pm [5]\setminus f'(v')=\{d_1,d_2,\pm d_3,d_4,\pm d_5\}$ where $|c_i|=|d_i|=i$ for every $i\in[5]$.  
		We define a mapping $g(x)\in {[5] \choose \pm (6-d_G(x))}$ for each $x\in V(G)$ as follows: 
		$$g(u)=\{c_1,c_3,-d_4,c_5\},~ g(v)=\{d_2,-c_3,d_4,-c_5\},$$ 
		
		$g(u')=f'(u')$, $g(v')=f'(v')$, and $g(x)=f(x)$ for $x\in V(\widehat{G})\setminus \{u,v,u',v'\}$. Noting that $g(u)\subset A_{f'}(u)$ and $g(v)\subset A_{f'}(v)$, one can easily check that $g$ is a $(5, \phi)$-coloring of $\widehat{G}$ where $\phi(x)=6-d_{G}(x)$, a contradiction. 
	\end{proof}
	
	Next, we consider $3$-vertices with neighbors of degree $2$ in $\widehat{G}$.
	
	\begin{claim}\label{claim:2-3-2}
		$\widehat{G}$ contains no $3$-vertex with two $2$-neighbors.
	\end{claim}
	
	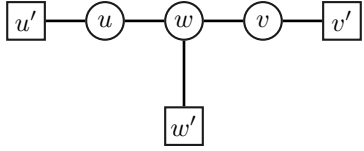
\begin{figure}[htbp]
		\centering
		\begin{tikzpicture}[>=latex,
			roundnode/.style={circle, draw=black!90, thick, minimum size=5mm, inner sep=0pt},
			squarenode/.style={rectangle, draw=black!90, thick, minimum size=5mm, inner sep=0pt},
			scale=0.7
			]
			\node [roundnode] (v) at (1.5,0) {$v$};
			\node [roundnode] (u) at (-1.5,0) {$u$};
			\node [roundnode] (w) at (0,0) {$w$};
			\node [squarenode] (w') at (0,-2) {$w'$};
			\node [squarenode] (y) at (3, 0) {$v'$};
			\node [squarenode] (x) at (-3, 0) {$u'$};
			
			\draw [line width =1pt, black] (y)--(v)--(w)--(u)--(x);
			\draw [line width =1pt, black] (w')--(w);
		\end{tikzpicture}
		\caption{A $3$-vertex having two $2$-neighbors}  
		\label{fig:2-3-2}  
	\end{figure}

	\begin{proof}
		Let $w$ be such a $3$-vertex in $\widehat{G}$ with two $2$-neighbors $u$ and $v$ with $N_G(u)=\{w,u'\}$, $N_G(v)=\{w,v'\}$, and $N_G(w)=\{u,v,w'\}$. See \Cref{fig:2-3-2}. By~\Cref{claim:3OR4cycles}, there is neither a $(2,3,3)$-triangle nor a $(2,3,2,3)$-cycle, so $w',u',$ and $v'$ are three distinct vertices. Moreover, by~\Cref{claim:2-2}, $d_G(u')=d_G(v')=3$. By switching, if needed, we may assume that $uw$, $vw$, and $ww'$ are all negative. 
		
		Let $\widehat{G'}:=\widehat{G}-\{w\}$. By \Cref{claim:3OR4cycles}, no block isomorphic to $C^*_3$ or $C^*_4$ is created in $\widehat{G'}$. By minimality, there exists a $(5, \phi')$-coloring $f$ of $\widehat{G'}$ where $\phi'(x)=6-d_{G'}(x)$ for $x \in V(G')$. As $d_{G'}(u)=d_{G'}(v)=1$ and each of them has a neighbor of degree $3$ in $\widehat{G'}$, by the definition of $f$, $|A_f(v)|=|A_f(u)|=7$ and $|A_f(v)^*|=|A_f(u)^*|=5$. Without loss of generality, we may assume that $A_f(v)=\{c_1,c_2,c_3,\pm 4, \pm 5\}$ and $A_f(u)=\{d_1,d_2,d_3,d_4,d_5, d^1_u, d^2_u\}$ where $|c_i|=|d_i|=i$ for every $i\in[5]$ and $d^j_u\in \pm[5]\setminus \{d_1,d_2,d_3,d_4,d_5\}$ for $j\in [2]$. Observing that $1\leq d_{G'}(w')\leq 2$, we consider the following cases based on the degree of $w'$. Let $a_i\in \pm [5]$ such that $|a_i|=i$ for $i\in [5]$.
		
		\medskip
		Case~(1) Assume that $d_G(w')=2$. In this case, we can assume $f(w')=\{a_1,a_2,a_3,a_4,a_5\}$. We consider two subcases:
		
		(1.1) If one of $|d^j_u|\in \{1,2,3\}$, say $d^j_u=-d_3$ (i.e., $\pm 3\in f(u)$), then we define $g(v)=\{c_1,c_2,d_4,a_5\}$, $g(u)=\{d_1,d_2,a_3,d_4\}$, $g(w')=\{a_1,a_2,a_3,a_5\}$ and $g(w)=\{-a_3,-d_4,-a_5\}$.
		
		(1.2) Otherwise, $f(u)=\{d_1,d_2,d_3,\pm 4, \pm 5\}$. Now we define $g(v)=\{c_1,c_2,a_4,a_5\}$, $g(u)=\{d_1,d_2,a_4,a_5\}$, $g(w')=\{a_1,a_2,a_4,a_5\}$ and $g(w)=\{a_3,-a_4,-a_5\}$.
		
		\medskip
		Case~(2) Assume that $d_G(w')=3$. In this case, $f(w')\in {[5] \choose \pm 4}$. As we fix $A_f(v)=\{c_1,c_2,c_3,\pm 4, \pm 5\}$, by symmetry there are only two possibilities for $f(w')$.
		
		(2.1) If $f(w')=\{a_1,a_2,a_3,a_4\}$, then define $g(v)=\{c_1,c_3,a_4,d_5\}$, $g(u)=\{d_1,d_2,d_3,d_5\}$, $g(w')=\{a_1,a_3,a_4\}$ and $g(w)=\{-d_2,-a_4,-d_5\}$. 
		
		(2.2) If $f(w')=\{a_1,a_2,a_4,a_5\}$, then define $g(v)=\{c_1,c_2,a_4,d_5\}$, $g(u)=\{d_1,d_2,d_3,d_5\}$, $g(w')=\{a_1,a_2,a_4\}$ and $g(w)=\{-d_3,-a_4,-d_5\}$. 
		
		\medskip
		In each case, let $g(x)=f(x)$ for $x\in V(\widehat{G})\setminus \{u,w,v,w'\}$. Note that $g(u)\subset f(u)$, $g(v)\subset f(v)$, and $g(w')\subset f(w')$, and, moreover, $g(w)\cap g(y)=\emptyset$ for each $y\in \{u,v,w'\}$. Thus, $g$ is a $(5, \phi)$-coloring of $\widehat{G}$ where $\phi(x)=6-d_{G}(x)$, a contradiction.
	\end{proof}

	We now consider triangles or $4$-cycles sharing edges in $\widehat{G}$. 
	
	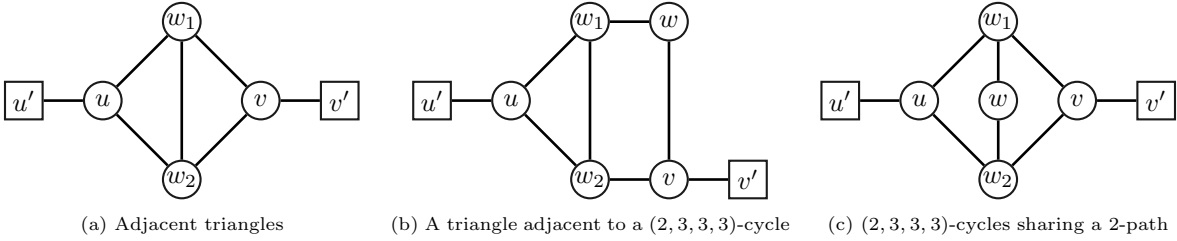
\begin{figure}[htbp]
		\centering  
		\begin{subfigure}[t]{.32\textwidth}
			\centering
			\begin{tikzpicture}[>=latex,
				roundnode/.style={circle, draw=black!90, thick, minimum size=5mm, inner sep=0pt},
				squarenode/.style={rectangle, draw=black!90, thick, minimum size=5mm, inner sep=0pt},
				scale=0.7
				]
				\node [roundnode] (v) at (1.5,0) {$v$};
				\node [roundnode] (u) at (-1.5,0) {$u$};
				\node [roundnode] (w1) at (0,1.5) {$w_1$};
				\node [roundnode] (w2) at (0,-1.5) {$w_2$};
				\node [squarenode] (y) at (3, 0) {$v'$};
				\node [squarenode] (x) at (-3, 0) {$u'$};
				
				\draw [line width =1pt, black] (x)--(u)--(w1)--(v);
				\draw [line width =1pt, black] (u)--(w2)--(v)--(y);
				\draw [line width =1pt, black] (w1)--(w2);
			\end{tikzpicture}
			\caption{Adjacent triangles}  
			\label{fig:adj-triangles}  
		\end{subfigure}
		\begin{subfigure}[t]{.32\textwidth}
			\centering
			\begin{tikzpicture}[>=latex,
				roundnode/.style={circle, draw=black!90, thick, minimum size=5mm, inner sep=0pt},
				squarenode/.style={rectangle, draw=black!90, thick, minimum size=5mm, inner sep=0pt},
				scale=0.7
				]
				\node [roundnode] (v) at (1.5,-1.5) {$v$};
				\node [roundnode] (w) at (1.5,1.5) {$w$};
				\node [roundnode] (u) at (-1.5,0) {$u$};
				\node [roundnode] (w1) at (0,1.5) {$w_1$};
				\node [roundnode] (w2) at (0,-1.5) {$w_2$};
				\node [squarenode] (y) at (3,-1.5) {$v'$};
				\node [squarenode] (x) at (-3, 0) {$u'$};
				
				\draw [line width =1pt, black] (x)--(u)--(w1)--(w2);
				\draw [line width =1pt, black] (u)--(w2)--(v)--(y);
				\draw [line width =1pt, black] (w1)--(w)--(v);
			\end{tikzpicture}
			\caption{A triangle adjacent to a $(2,3,3,3)$-cycle}  
			\label{fig:triangle-2333-cycle} 
		\end{subfigure}
		\begin{subfigure}[t]{.32\textwidth}
			\centering
			\begin{tikzpicture}[>=latex,
				roundnode/.style={circle, draw=black!90, thick, minimum size=5mm, inner sep=0pt},
				squarenode/.style={rectangle, draw=black!90, thick, minimum size=5mm, inner sep=0pt},
				scale=0.7
				]
				\node [roundnode] (v) at (1.5,0) {$v$};
				\node [roundnode] (u) at (-1.5,0) {$u$};
				\node [roundnode] (w1) at (0,1.5) {$w_1$};
				\node [roundnode] (w2) at (0,-1.5) {$w_2$};
				\node [roundnode] (w) at (0,0) {$w$};
				\node [squarenode] (y) at (3, 0) {$v'$};
				\node [squarenode] (x) at (-3, 0) {$u'$};
				
				\draw [line width =1pt, black] (x)--(u)--(w1)--(v);
				\draw [line width =1pt, black] (u)--(w2)--(v)--(y);
				\draw [line width =1pt, black] (w1)--(w)--(w2);
			\end{tikzpicture}
			\caption{$(2,3,3,3)$-cycles sharing a $2$-path}  
			\label{fig:two-2333-cycle} 
		\end{subfigure}
		\caption{Configurations in \Cref{claim:adjacentT}}
		\label{fig:adjacent-cycles}
	\end{figure}

	\begin{claim}\label{claim:adjacentT}
		$\widehat{G}$ contains none of the following: adjacent triangles, a triangle adjacent to a $(2,3,3,3)$-cycle, or two $(2,3,3,3)$-cycles sharing a $2$-path.
	\end{claim}
	
	\begin{proof}
		Assume to the contrary that $\widehat{G}$ contains one of the said structures. We use the labeling of~\Cref{fig:adj-triangles}, \Cref{fig:triangle-2333-cycle} and \Cref{fig:two-2333-cycle}. In each case, since $\widehat{G}$ has at least 5 vertices, $\widehat{G}$ has no cut-vertex and so $u'\neq v'$. Moreover, observe that $u'$ and $v'$ cannot be in the same $(2,3,3,3)$-cycle or the same $(3,3,3)$-triangle in $\widehat{G}$. 
		
		Let $X_1=\{u,v,w_1,w_2\}$ and $X_2=X_3=\{u,v,w,w_1,w_2\}$. Let $\widehat{G'_i}:=\widehat{G}-X_i$ for $i\in [3]$. By~\Cref{claim:3OR4cycles} and the above observation, $\widehat{G'}_i$ contains no block isomorphic to $C^*_3$ or $C^*_4$. By minimality, there exists a $(5, \phi')$-coloring $f$ of $\widehat{G'_i}$ where $\phi'(x)=6-d_{G'_i}(x)$ for $x \in V(G'_i)$. Since $d_G(u')=d_{G'_i}(u')+1$ and $d_G(v')=d_{G'_i}(v')+1$, by definition $f(u')\in {[5] \choose \pm (7-d_G(u'))}$ and $f(v')\in {[5] \choose \pm (7-d_G(v'))}$. By \Cref{obs:proper_subset}, we can choose two subsets $f'(v')\subsetneq f(v')$ and $f'(u')\subsetneq f(u')$ satisfying that 
		$f'(v')\in {[5] \choose \pm (6-d_G(v'))}$, $f'(u')\in {[5] \choose \pm (6-d_G(u'))}$, and $f'(v')\neq f'(u')$. Let $A_{f'}(u)=\pm[5]\setminus f'(u')$ and $A_{f'}(v)=\pm[5]\setminus f'(v')$. Without loss of generality, assume that $\{c_1,c_2,c_3,c_4,\pm c_5\}\subseteq A_{f'}(u)$ and $\{d_1,d_2,d_3,\pm d_4,d_5\}\subseteq A_{f'}(v)$, where $|c_i|=|d_i|=i$ for every $i\in[5]$.
		
		\bigskip
		We first consider $\widehat{G'_1}$ and define an appropriate mapping $g(x)\in {[5] \choose \pm (6-d_G(x))}$ for $x\in X_1$ such that $g(u)\subset A_{f'}(u)$ and $g(v)\subset A_{f'}(v)$. By possibly a switching, we may assume that $uw_1,uw_2$, and $vw_1$ are negative. 
		Based on the signature of two adjacent triangles, and by symmetry we have three cases to discuss. 
		
		\medskip
		(1) If $w_1w_2$ is negative and $vw_2$ is positive, then let $g(u)=\{c_1,c_4,d_5\}$, $g(v)=\{d_1,-c_4,d_5\}$, $g(w_1)=\{-c_2,-c_3, -d_5\}$, and $g(w_2)=\{c_2,c_3,-c_4\}$.
		
		\medskip
		(2) If $w_1w_2$ and $vw_2$ are both negative, then let $g(u)=\{c_1,c_4,d_5\}$, $g(v)=\{d_1,c_4,d_5\}$, $g(w_1)=\{-c_2,-c_3,-d_5\}$ and $g(w_2)=\{c_2,c_3,-c_4\}$. 
		
		\medskip
		(3) If $w_1w_2$ is positive and $vw_2$ is negative, then let $g(u)=\{c_1,c_4,d_5\}$, $g(v)=\{d_1,c_4,d_5\}$, and $g(w_1)=g(w_2)=\{-c_2,-c_3,-d_5\}$.

		\bigskip
		We next consider $\widehat{G'_2}$ and define an appropriate mapping $g(x)\in {[5] \choose \pm (6-d_G(x))}$ for $x\in X_2$ such that $g(u)\subset A_{f'}(u)$ and $g(v)\subset A_{f'}(v)$. By possibly a switching, we may assume that $uw_1, ww_1$, and $vw_2$ are negative.
		Based on the signatures of the triangle and the $4$-cycle, noting that $\{uw_2, w_1w_2, wv\}$ is contained in an edge-cut, by possible switching we have four cases. 
		
		\medskip
		(1) If $uw_2, w_1w_2, wv$ are all negative, then we define $g(u)=\{c_1,c_4,d_5\}$, $g(v)=\{d_1,c_4,d_5\}$, $g(w_1)=\{-c_2,-c_3,-d_5\}$, $g(w_2)=\{c_2,c_3,-c_4\}$, and $g(w)=\{-d_1,c_2,c_3,-c_4\}$.

		\medskip
		(2) If $uw_2, w_1w_2$ are negative and $wv$ is positive, then let $g(u)=\{c_1,c_4,d_5\}$, $g(v)=\{d_1,c_4,d_5\}$, $g(w_1)=\{-c_2,-c_3,-d_5\}$, $g(w_2)=\{c_2,c_3,-c_4\}$, and $g(w)=\{d_1,c_2,c_3,d_5\}$.  
		
		\medskip
		(3) If $uw_2, wv$ are negative and $w_1w_2$ is positive, then let $g(u)=\{c_1,c_4,d_5\}$, $g(v)=\{d_1,c_4,d_5\}$, $g(w_1)=g(w_2)=\{-c_2,-c_3,-d_5\}$, and $g(w)=\{-d_1,c_2,c_3,-c_4\}$. 
		
		\medskip
		(4) If $w_1w_2, wv$ are negative and $uw_2$ is positive, then let $g(u)=\{c_1,c_4,d_5\}$, $g(v)=\{d_1,-c_4,d_5\}$, $g(w_1)=\{-c_2,-c_3,-d_5\}$, $g(w_2)=\{c_2,c_3,c_4\}$, and $g(w)=\{-d_1,c_2,c_3,c_4\}$.

		\bigskip
		We finally consider $\widehat{G'_3}$ and define an appropriate mapping $g(x)\in {[5] \choose \pm (6-d_G(x))}$ for $x\in X_3$ such that $g(u)\subset A_{f'}(u)$ and $g(v)\subset A_{f'}(v)$. By possible switching, we may assume that $uw_1, ww_1, vw_1$, and $uw_2$ are negative.
		Based on the signatures of the two $4$-cycles, by possible switching we have three cases. 
		
		\medskip
		(1) If $ww_2$ and $vw_2$ are negative, then we define $g(u)=\{c_1,c_4,d_5\}$, $g(v)=\{d_1,c_4,d_5\}$, $g(w_1)=g(w_2)=\{c_2,c_3,-d_5\}$, and $g(w)=\{-c_2,-c_3,c_4,d_5\}$.

		\medskip
		(2) If $ww_2$ is positive and $vw_2$ is negative, then let $g(u)=\{c_1,c_4,d_5\}$, $g(v)=\{d_1,c_4,d_5\}$, $g(w_1)=\{c_2,c_3,-d_5\}$, $g(w_2)=\{-c_2,-c_3,-d_5\}$, and $g(w)=\{c_1,-c_2,-c_3,c_4\}$.  
		
		\medskip
		(3) If $ww_2$ is negative and $vw_2$ is positive, then let $g(u)=\{c_1,c_4,d_5\}$, $g(v)=\{d_1,-c_4,d_5\}$, $g(w_1)=\{c_2,c_3,-d_5\}$, $g(w_2)=\{c_2,c_3,-c_4\}$ and $g(w)=\{-c_2,-c_3,c_4,d_5\}$.

		\bigskip
		For each case, together with $g(u')=f'(u')$, $g(v')=f'(v')$, and $g(x)=f(x)$ for $x\in V(\widehat{G})\setminus \{u,v,u',v',w_1,w_2\}$, each mapping $g$ is a $(5, \phi)$-coloring of $\widehat{G}$ with $\phi(x)=6-d_{G}(x)$, a contradiction. 
	\end{proof}

	\begin{claim}\label{claim:2-3}
		$\widehat{G}$ contains no $3$-vertex with one $2$-neighbor. Consequently, $\widehat{G}$ is cubic.
	\end{claim}
	
	\begin{proof}
		Let $v$ be a $3$-vertex of $\widehat{G}$ and $N_G(v)=\{v_1,v_2,v_3\}$. By \Cref{claim:2-3-2}, at most one of $v_i$ is a $2$-vertex and assume that $d_G(v_1)=2$. We may apply switching such that, for $i\in[3]$, each $vv_i$ is negative in $\widehat{G}$. Let $\widehat{G'}:=\widehat{G}-\{v\}$.
		By \Cref{claim:adjacentT}, no block isomorphic to $C^*_3$ or $C^*_4$ is created by deleting a $3$-vertex, and thus $\widehat{G'}$ contains no block isomorphic to any member of $\mathcal{B}_0$.
		
		By minimality, $\widehat{G'}$ admits a $(5, \phi')$-coloring $f$ of $\widehat{G'}$ where $\phi'(x)=6-d_{G'}(x)$ for $x \in V(G')$. Since $d_{G'}(v_i)=1$ and $d_{G'}(v_i)=2$ for $i\in \{2,3\}$, we have $f(v_1)\in {[5] \choose \pm 5}$ and $f(v_i)\in {[5] \choose \pm 4}$ for $i\in \{2,3\}$. Without loss of generality, assume that $\{a_1,a_2,a_3,a_4,a_5\}\subset f(v_1)$ and $f(v_2)=\{b_1,b_2,b_3,b_4\}$. By symmetry, we may assume that $f(v_3)=\{c_1,c_2,c_3,c_k\}$ for some $k\in \{4,5\}$. Note that $|a_i|=|b_i|=|c_i|=i$ for each $i\in [5]$. We consider two cases:
		
		\medskip
		Case (1) If one of $b_i=c_i$ for $i\in [3]$, say $b_1=c_1$, then we define $g(v_1)=\{a_2,a_3,a_4,a_5\}$, $g(v_2)=\{b_1,b_2,b_3\}$, $g(v_3)=\{b_1, c_2,c_3\}$, and $g(v)=\{-b_1, -a_4,-a_5\}$.

		\medskip
		Case (2) Otherwise, for each $i\in [3]$, $b_i\neq c_i$. So $a_1$ is the same as exactly one of $b_1$ or $c_1$. 
		
		(2.1) If $a_1=b_1$, then we define $g(v_1)=\{a_1,a_2,a_3,a_{9-k}\}$, $g(v_2)=\{a_1,b_2,b_3\}$, $g(v_3)=\{c_2,c_3,c_k\}$, and $g(v)=\{-a_1,-a_{9-k},-c_k\}$. 
		
		(2.2) If $a_1=c_1$, then we define $g(v_1)=\{a_1,a_2,a_3,a_5\}$, $g(v_2)=\{b_2,b_3,b_4\}$, $g(v_3)=\{a_1,c_2,c_3\}$, and $g(v)=\{-a_1,-b_4,-a_5\}$. 
		
		\medskip
		For each case, let $g(x)=f(x)$ for $x\in V(\widehat{G})\setminus \{v,v_1,v_2,v_3\}$. Each mapping $g$ is a $(5, \phi)$-coloring of $\widehat{G}$ with $\phi(x)=6-d_{G}(x)$, a contradiction.  
	\end{proof}
	
	We complete the proof with our last claim. It contradicts the fact that $\widehat{G}$ must contain a $3$-vertex. Recall that $\widehat{G}$ has at least $5$ vertices.

	\begin{claim}\label{claim:contradiction}
		$\widehat{G}$ contains no $3$-vertices.
	\end{claim}
	
	\begin{proof}
		Let $v$ be a $3$-vertex of $\widehat{G}$ and $N_G(v)=\{v_1,v_2,v_3\}$. By \Cref{claim:2-3}, $d_G(v_i)=3$ for each $i\in [3]$. We may apply switching such that each $vv_i$ for $i\in[3]$ is negative in $\widehat{G}$. Let $\widehat{G'}:=\widehat{G}-\{v\}$.
		By~\Cref{claim:3OR4cycles} and \Cref{claim:adjacentT}, $\widehat{G'}$ contains no block isomorphic to $C^*_3$ or $C^*_4$. 
		
		By minimality, $\widehat{G'}$ admits a $(5, \phi')$-coloring $f$ of $\widehat{G'}$ where $\phi'(x)=6-d_{G'}(x)$ for $x \in V(G')$. Since $d_{G'}(v_i)=2$ for $i\in [3]$, we have $f(v_i)\in {[5] \choose \pm 4}$ for $i\in [3]$. Without loss of generality, we may assume that $f(v_1)=\{a_1,a_2,a_3,a_4\}$ with $|a_i|=i$ for each $i\in [4]$. By symmetry, we have two possibilities for $f(v_2)$: (1) $f(v_2)=\{b_1,b_2,b_3,b_4\}$; (2) $f(v_2)=\{b_\alpha,b_\beta,b_\gamma,b_5\}$ where $|b_i|=i$. Moreover, assume that $f(v_3)=\{c_h, c_j, c_k, c_\ell\}$ with $|c_i|=i$ for $i\in \{h,j,k,\ell\}$. We discuss two cases:
		
		\medskip
		{\bf Case (1).} $f(v_2)=\{b_1,b_2,b_3,b_4\}$.
		
		\medskip
		{\bf Subcase (1.1)} \emph{One of $\{c_h, c_j, c_k, c_\ell\}$ has absolute value $5$, say $|c_\ell|=5$.} 
		
		In this case, $\{|c_h|, |c_j|, |c_k|\}\subset [4]$. Without loss of generality, we may assume that $\{c_h, c_j, c_k, c_\ell\}=\{c_1,c_2,c_3,c_5\}$. We define a new mapping $g$ as follows: $$g(v_1)=\{a_1,a_2,a_3\},~g(v_2)=\{b_1,b_2,b_4\},~g(v_3)=\{c_1,c_2,c_5\},~g(v)=\{-a_3,-b_4,-c_5\}$$ and $g(x)=f(x)$ for each $x\in V(\widehat{G})\setminus \{v,v_1,v_2,v_3\}$. Noting that $g(v_i)\subset f(v_i)$ for $i\in [3]$, this mapping $g$ is a $(5, \phi)$-coloring of $\widehat{G}$ where $\phi(x)=6-d_{G}(x)$ for $x \in V(G)$.
		
		In the following subcases, we know that $f(v_3)\cap \{5,-5\}=\emptyset$ and may thus assume that $\{c_h, c_j, c_k, c_\ell\}=\{c_1,c_2,c_3,c_4\}$.
		
		\medskip
		{\bf Subcase (1.2)} \emph{One of $\{c_1,c_2,c_3,c_4\}$, say $c_1$, belongs to $f(v_1)\cap f(v_2)$.} 
		
		That is to say, $a_1=b_1=c_1$. We define a mapping $g$ as follows: $$g(v_1)=\{a_1,a_2,a_3\},~g(v_2)=\{b_1,b_2,b_3\},~g(v_3)=\{c_1,c_2,c_3\},~g(v)=\{-a_1,4,5\}$$ and $g(x)=f(x)$ for each $x\in V(\widehat{G})\setminus \{v,v_1,v_2,v_3\}$. One may readily check that $g$ is a $(5, \phi)$-coloring of $\widehat{G}$ where $\phi(x)=6-d_{G}(x)$ for $x \in V(G)$.
		
		\medskip
		{\bf Subcase (1.3)} \emph{Otherwise, each $c_i$ for $i\in \{1,2,3,4\}$ is in either $\{i, -i\}\setminus \{a_i, b_i\}$ or $(f(v_1)\setminus f(v_2))\cup (f(v_2)\setminus f(v_1))$.} 
		
		Since each set $f(v_j)$ for $j\in [3]$ is in the same form, without loss of generality, we may assume that $a_1=b_1=-c_1$. We define a mapping $g$ as follows: $$g(v_1)=\{a_1,a_2,a_3\},~g(v_2)=\{b_1,b_2,b_3\},~g(v_3)=\{c_2,c_3,c_4\},~g(v)=\{c_1, -c_4, 5\}$$ and $g(x)=f(x)$ for each $x\in V(\widehat{G})\setminus \{v,v_1,v_2,v_3\}$. Such a mapping $g$ is a $(5, \phi)$-coloring of $\widehat{G}$ where $\phi(x)=6-d_{G}(x)$ for $x \in V(G)$.

		\medskip
		{\bf Case (2).} $f(v_2)=\{b_\alpha,b_\beta,b_\gamma,b_5\}$ for $\alpha,\beta,\gamma\in [4]$.
		
		In this case, we first claim that $\{h,j,k,\ell\}\neq \{1,2,3,4\}$ and $\{h,j,k,\ell\}\neq \{\alpha,\beta,\gamma,5\}$. As otherwise, by the symmetries of the vertices $v_1,v_2,$ and $v_3$, we may apply the same arguments in Subcase (1.1) and complete the proof. 
		Recalling that $\{\alpha,\beta,\gamma\}$ is a $3$-subset of $\{1,2,3,4\}$, we have that $\{c_h, c_j, c_k, c_\ell\}=\{c_{\alpha'}, c_{\beta'}, c_{\gamma'}, c_5\}$ where $\alpha', \beta', \gamma'\in [4]$ and $|\{\alpha,\beta,\gamma\}\cap \{\alpha',\beta',\gamma'\}|=2$. Without loss of generality, we assume that $\alpha=\alpha', \beta=\beta'$ and $\gamma\neq \gamma'$.
		We define a mapping $g$ as follows: $$g(v_1)=\{a_\alpha,a_\beta,a_\gamma\},~g(v_2)=\{b_\alpha,b_\beta,b_5\},~g(v_3)=\{c_{\alpha'},c_{\beta'},c_{\gamma'}\},~g(v)=\{-a_\gamma, -b_5, -c_{\gamma'}\}$$ and $g(x)=f(x)$ for each $x\in V(\widehat{G})\setminus \{v,v_1,v_2,v_3\}$. Since $\gamma\neq \gamma'$ and $\gamma, \gamma'\in [4]$, $g$ can be checked to be a $(5, \phi)$-coloring of $\widehat{G}$ where $\phi(x)=6-d_{G}(x)$ for $x \in V(G)$.
	\end{proof}

	\section{Remarks and questions}\label{sec:Que}
	Besides the example $\widehat{K}_4^{\bullet}$, it is unclear whether there exists an infinite family of subcubic signed graphs critically having fractional balanced chromatic number $\frac{5}{3}$. 
	
	We conjecture that if we further forbid $\widehat{K}_4^{\bullet}$ from our graph class of \Cref{thm:fractional-main}, the fractional balanced chromatic number is bounded by $\frac{8}{5}$.
	\begin{conjecture}\label{conj:8/5}
		Every signed subcubic graph not isomorphic to $(K_4,-)$ and not containing $\widehat{K}_4^{\bullet}$ admits an $(8,5)$-coloring.
	\end{conjecture}
	
	The value of $\frac{8}{5}$ is reached by a signed $3$-dimensional hypercube depicted in \Cref{fig:NegCube}. One may ask the same question of finding infinite critical examples that achieve this bound.
	
	\begin{figure}[!htbp]
		\centering
		
		\centering
		\begin{tikzpicture}[scale=.45]		
			\draw [line width=0.4mm, dotted, blue] (-2,-2) to (2,-2);
			\draw [line width=0.4mm, dotted, blue] (-2,-2) to (-2,2);
			\draw [line width=0.4mm, dotted, blue] (2,-2) to (2,2);
			\draw [line width=0.4mm, red] (-2,2) to (2,2);
			
			\draw [line width=0.4mm, dotted, blue] (-4,-4) to (4,-4);
			\draw [line width=0.4mm, dotted, blue] (-4,-4) to (-4,4);
			\draw [line width=0.4mm, red] (4,-4) to (4,4);
			\draw [line width=0.4mm, dotted, blue] (-4,4) to (4,4);
			
			\draw [line width=0.4mm, red] (-2,-2) to (-4,-4);
			\draw [line width=0.4mm, dotted, blue] (2,2) to (4,4);
			\draw [line width=0.4mm, dotted, blue] (2,-2) to (4,-4);
			\draw [line width=0.4mm, dotted, blue] (-2,2) to (-4,4);
			
			\draw [fill=white,line width=0.5pt] (2,-2) node[above] {} circle (4pt); 
			\draw [fill=white,line width=0.5pt] (-2,2) node[above] {} circle (4pt); 
			\draw [fill=white,line width=0.5pt] (2,2) node[above] {} circle (4pt); 
			\draw [fill=white,line width=0.5pt] (-2,-2) node[above] {} circle (4pt); 
			\draw [fill=white,line width=0.5pt] (4,-4) node[above] {} circle (4pt); 
			\draw [fill=white,line width=0.5pt] (-4,4) node[above] {} circle (4pt); 
			\draw [fill=white,line width=0.5pt] (4,4) node[above] {} circle (4pt); 
			\draw [fill=white,line width=0.5pt] (-4,-4) node[above] {} circle (4pt); 
			
		\end{tikzpicture}
		\caption{A $3$-cube with all faces negative}
		\label{fig:NegCube}	
		
	\end{figure}
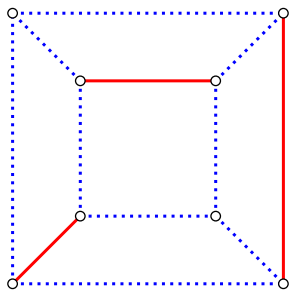
	
	\section*{Acknowledgment} This work has received support under the program ``Investissement d'Avenir" launched by the French Government and implemented by ANR, with the reference ``ANR‐18‐IdEx‐0001" as part of its program ``Emergence".  Xiaolan Hu is supported by NSFC (No. 12471325). Jiaao Li is supported by National Key Research and Development Program of China (No. 2022YFA1006400), National Natural Science Foundation of China (Nos. 12222108, 12131013), Natural Science Foundation of Tianjin (Nos. 24JCJQJC00130, 22JCYBJC01520), and the Fundamental Research Funds for the Central Universities, Nankai University. Lujia Wang is supported by Natural Science Foundation of China (No. 12371359). Zhouningxin Wang is partially supported by National Natural Science Foundation of China (No. 12301444) and the Fundamental Research Funds for the Central Universities, Nankai University. Xiaowei Yu is supported by the Basic Research Program Young Science and Technology Talent Project of Xuzhou (KC23026).

\end{document}